\documentclass{amsart} 
\usepackage{latexsym}
\usepackage{amsfonts}
\usepackage{amssymb}
\usepackage{amsmath}
\usepackage{amscd}
\usepackage{graphicx}
\usepackage{eucal}
\usepackage{amsthm}
\usepackage[all]{xy}
\usepackage{pdfsync}
\usepackage{xspace}
\usepackage[unicode=true, pdfusetitle,
 bookmarks=true,bookmarksnumbered=false,
 breaklinks=false,
 backref=false,
 colorlinks=true,
 linkcolor=blue,
 citecolor=black,
 urlcolor=blue!78!red,
 final
]{hyperref}
\usepackage[capitalise]{cleveref}

\newcommand{\newrefformat}[2]{}

\crefname{lem}{Lemma}{Lemmas}
\crefname{thm}{Theorem}{Theorems}
\crefname{Defi}{Definition}{Definitions}
\crefname{nota}{Notation}{Notations}
\crefname{construction}{Construction}{Constructions}
\crefname{prop}{Proposition}{Propositions}
\crefname{rem}{Remark}{Remarks}
\crefname{cor}{Corollary}{Corollaries}
\crefname{scholium}{Scholium}{Scholia}
\crefname{figure}{Figure}{Figures}
\crefname{equation}{Display}{Displays}
\crefname{eq}{Display}{Displays}
\crefname{eqn}{Display}{Displays}

\newenvironment{eq*}{\begin{equation*}}{\end{equation*}}
\newenvironment{eqn*}{\begin{equation*}}{\end{equation*}}

\newcommand{\mb}{\mathbf}

\newcommand{\m}[1]{\mathcal{#1}}

\newcommand{\g}[1]{\ensuremath{\mathbb{#1}}\xspace}
\newcommand{\Cat}{\ensuremath{\textnormal{Cat}}\xspace}
\newcommand{\cd}[2][]{\vcenter{\hbox{\xymatrix#1{#2}}}}
\newcommand{\Set}{\ensuremath{\textnormal{Set}}\xspace}
\newcommand{\twocat}{\ensuremath{\textnormal{2-Cat}}\xspace}
\newcommand{\Icon}{\ensuremath{\textnormal{Icon}}\xspace}
\newcommand{\TAlgs}{\ensuremath{\textnormal{T-Alg}_{\textnormal{s}}}\xspace}
\newcommand{\TAlgo}{\ensuremath{\textnormal{T-Alg}_{\textnormal{o}}}\xspace}
\newcommand{\id}{\textrm{id}}

\begin{document}
\numberwithin{equation}{section}
\newtheorem{thm}[equation]{Theorem}
\newtheorem{prop}[equation]{Proposition}
\newtheorem{lem}[equation]{Lemma}
\newtheorem{cor}[equation]{Corollary}

\newtheoremstyle{example}{\topsep}{\topsep}%
     {}
     {}
     {\bfseries}
     {.}
     {2pt}
     {\thmname{#1}\thmnumber{ #2}\thmnote{ #3}}

   \theoremstyle{example}
   \newtheorem{nota}[equation]{Notation}
   \newtheorem{example}[equation]{Example}
   \newtheorem{Defi}[equation]{Definition}
   \newtheorem{rem}[equation]{Remark}
	\newtheorem{comment}[equation]{Comment}
   \newtheorem{conv}[equation]{Convention}

   \title[Gray]{The Gray tensor product via factorisation}

\author{John Bourke}
\address{Department of Mathematics and Statistics, Masaryk University, Kotl\'a\v rsk\'a 2, Brno 60000, Czech Republic}
\email{bourkej@math.muni.cz}
\curraddr{Department of Mathematics, Macquarie University, NSW 2109 (Australia)}
\email{john.d.bourke@mq.edu.au}

\author{Nick Gurski}
\address{
School of Mathematics and Statistics,
University of Sheffield,
Sheffield, UK, S3 7RH
}
\email{nick.gurski@sheffield.ac.uk}

\begin{abstract}
We discuss the folklore construction of the Gray tensor product of 2-categories as obtained by factoring the map from the funny tensor product to the cartesian product.  We show that this factorisation can be obtained without using a concrete presentation of the Gray tensor product, but merely its defining universal property, and use it to give another proof that the Gray tensor product forms part of a symmetric monoidal structure. 
The main technical tool is a method of producing new algebra structures over Lawvere 2-theories from old ones via a factorisation system.
\end{abstract}

\maketitle

\tableofcontents
	
\section*{Introduction}

The Gray tensor product of 2-categories, denoted $\otimes$, plays a central role in 2- and 3-dimensional category theory.  It is the tensor product for a closed, symmetric monoidal structure on the category of 2-categories and (strict) 2-functors, with the internal hom $\mb{Ps}(A,B)$ given by the 2-category of 2-functors, \textit{pseudo}natural transformations, and modifications with source $A$ and target $B$.  As pseudonatural transformations are crucial in much of 2-category theory, so are the 2-categories of the form  $\mb{Ps}(A,B)$ and hence by adjunction the Gray tensor product.  In 3-dimensional category theory, the importance of the Gray tensor product stems from the fact that categories enriched in $(\twocat, \otimes)$ are a convenient model for fully weak 3-categories.

There is a notable hurdle to overcome when working with the Gray tensor product, namely that one can define it fairly simply in terms of its universal property, but giving a concrete description can only be accomplished using a generators-and-relations approach.  While one can use this description to give a normal form for the 1- and 2-cells in $A \otimes B$, this does not make the task of proving that $\otimes$ is part of a symmetric monoidal structure any more tractable.  Gray's original proof\footnote{We note that Gray originally studied the lax version of $\otimes$ as opposed to the pseudo version we use here.} of the pentagon axiom uses some quite complicated calculations in the positive braid monoids \cite{Gray1976Coherence}.  The quickest proof of which the authors are aware\footnote{We briefly describe this strategy.  Firstly, one calculates that the homs $Ps(A,B)$ equip \twocat with a symmetric closed structure.   Secondly, one verifies that each functor $Ps(A,-)$ has a left adjoint $- \otimes A$.  Now a consequence of Proposition 2.3 of \cite{Day1978On} is that a symmetric closed category $(C,[-,-],I)$ together with adjunctions $- \otimes A \dashv [A,-]$ gives rise to a symmetric monoidal category $(C,\otimes,I)$.  This gives the sought after symmetric monoidal structure on \twocat.} uses the machinery of symmetric closed categories; this approach reduces the proof to some basic, but quite tedious, calculations involving pseudonatural transformations.

This paper originated in our desire to give another construction of the symmetric monoidal structure.  This is based upon the folklore fact,  due to Steve Lack, that the Gray tensor product $A \otimes B$ appears as the middle term in a factorisation
\[
A \star B \rightarrow A \otimes B \rightarrow A \times B
\]
from the funny tensor product to the cartesian product, in which the first map is an isomorphism on underlying categories and the second map is locally full and faithful.  One must then show that such a factorisation of a comparison map from one monoidal structure to another produces a new monoidal structure, in this case that for $\otimes$.  The second author gave a proof of such a factorisation in \cite{Gurski2013Coherence} using the generators-and-relations definition of the Gray tensor product, but that argument is not particularly insightful nor does it have wider applicability.

This paper has three sections.  The first is an entirely self-contained section giving the main technical tool, a method for producing new algebra structures from old ones using a factorisation system.  We phrase this in terms of Lawvere 2-theories, and believe this result will be of independent interest.  The second section gives the necessary background on the three tensor products we will study: $\star, \otimes$ and $ \times$.  Most notably, we will only develop $\otimes$ using its universal properties, completely avoiding the generators-and-relations description.  In the third section, we verify the existence of the necessary factorisation, thus completing the proof that $\otimes$ is part of a symmetric monoidal structure.

The second author was supported by EPSRC EP/K007343/1.

\section{Lawvere 2-theories and factorisation systems}

This section is devoted to proving a technical result that will enable us to produce, from
\begin{itemize}
\item a pair of monoidal structures $\otimes_{1}, \otimes_{2}$ on the same category $C$,
\item natural comparison maps $x \otimes_{1} y \to x \otimes_{2} y$ (and similarly with the units) and
\item an orthogonal factorisation system $(\m E,\m M)$ on $ C$ suitably compatible with both monoidal structures,
\end{itemize}
a new monoidal structure $\otimes$ and comparison maps
\[
x \otimes_{1} y \to x \otimes y \to x \otimes_{2} y.
\]
We give some preliminary definitions in order to state the desired result.

\begin{nota}
We will write monoidal categories with tensor product $\otimes$ and unit $I$ as $( C, \otimes, I)$, or $( C, \otimes, I, a, l, r)$ if we want to emphasize the associativity and unit isomorphisms.
\end{nota}

\begin{Defi}\label{Defi:oplax}
Let $( C, \otimes, I), ( D, \oplus, J)$ be monoidal categories.  An \emph{oplax monoidal functor} $F: C \to  D$ consists of a functor $F$ of the underlying categories, a morphism $\phi_{0}:F(I) \to J$, and a natural transformation $\phi_{-,-}$ with components
\[
\phi_{x,y}:F(x \otimes y) \to Fx \oplus Fy
\]
such that the following diagrams commute for all $x,y,z\in  C$.
\[
  \xy
  (0,0)*+{F(I \otimes x)}="00";
  (30,0)*+{FI \oplus Fx}="10";
  (30,-10)*+{J \oplus Fx}="11";
  (30,-20)*+{Fx}="12";
  {\ar^{\phi} "00"; "10"};
  {\ar^{\phi_{0} \oplus 1} "10"; "11"};
  {\ar^{l} "11"; "12"};
  {\ar_{Fl} "00"; "12"};
	(50,0)*+{F(x \otimes I)}="20";
  (80,0)*+{Fx \oplus FI}="30";
  (80,-10)*+{Fx \oplus J}="31";
  (80,-20)*+{Fx}="32";
  {\ar^{\phi} "20"; "30"};
  {\ar^{1 \oplus \phi_{0}} "30"; "31"};
  {\ar^{r} "31"; "32"};
  {\ar_{Fr} "20"; "32"};
	(20,-30)*+{F((x \otimes y) \otimes z)}="55";
	(60,-30)*+{F(x \otimes (y \otimes z))}="57";
	(100,-30)*+{Fx \oplus F(y \otimes z)}="56";
  (20,-45)*+{F(x \otimes y) \oplus Fz}="65";
  (60,-45)*+{(Fx \oplus Fy) \oplus Fz}="66";
	(100,-45)*+{Fx \oplus (Fy \oplus Fz)}="67";
	{\ar^{Fa} "55"; "57"};
	{\ar^{\phi} "57"; "56"};
  {\ar^{1 \oplus \phi} "56"; "67"};
  {\ar_{\phi} "55"; "65"};
  {\ar_{\phi \oplus 1} "65"; "66"};
	{\ar_{a} "66"; "67"};
  \endxy
  \]
If $ C,  D$ are symmetric with symmetries generically denoted $\beta$, then $F$ is oplax symmetric monoidal if the following additional axiom holds.
\[
\xy
(0,0)*+{F(x \otimes y) }="85";
  (40,0)*+{Fx \oplus Fy }="95";
  (40,-15)*+{Fy \oplus Fx}="86";
  (0,-15)*+{F(y \otimes x )}="96";
  {\ar^{\phi } "85"; "95"};
  {\ar^{\beta} "95"; "86"};
  {\ar_{F(\beta)} "85"; "96"};
  {\ar_{\phi} "96"; "86"};
  \endxy
  \]
\end{Defi}

\begin{rem}\label{rem:oplax-id}
It will be useful for us to record the case when the functor $F$ is actually the identity functor between two monoidal structures on the same category.  We will denote these as $( C, \otimes, I, a, l, r)$ and $( C, \otimes', I', a', l', r')$.  Then an oplax structure on the identity functor, viewed as a functor $( C, \otimes) \to ( C, \otimes')$, has data consisting of a natural transformation with components $x \otimes y \to x \otimes' y$ and a morphism $I \to I'$.  The axioms are then given as above, with $\oplus$ replaced by $\otimes'$, instances of $Fa$ replaced by $a$, instances of $a$ replaced by $a'$, similarly for $l$ and $r$, and all instances of the functor $F$ removed.  We leave this to the reader.
\end{rem}

Next we recall the notion of an orthogonal factorisation system on a category.

\begin{Defi}
A \emph{lifting problem} in a category  $ C$ with two chosen classes of maps $(\m E,\m M)$ is a  square
\[
\cd{A \ar[r]^{e} \ar[d] & B \ar[d]  \\
C \ar[r]_{m} & D}
\]
where $e \in \m E$ and $m \in \m M$.  A \emph{solution to a lifting problem} is a diagonal filler of the square such that both triangles commute.
\[
\cd{A \ar[r]^{e} \ar[d] & B \ar[d] \ar@{.>}[dl] \\
C \ar[r]_{m} & D}
\]
We will say that \emph{$e$ is orthogonal to $m$}, written $e \perp m$, if every lifting problem with horizontal arrows $e,m$ as above has a unique solution.
\end{Defi}

\begin{Defi}\label{Defi:ofs}
An \emph{orthogonal factorisation system} $(\m E,\m M)$ on $ C$ consists of two classes of maps, $\m E$ and $\m M$, such that
\begin{itemize}
\item every morphism $f$ in $C$ factors into a morphism in $\m E$ followed by a morphism in $\m M$, 
\item $\m E$ consists of all the morphisms $f$ such that $f \perp m$ for all $m \in \m M$ and
\item $\m M$ consists of all the morphisms $g$ such that $e \perp g$ for all $e \in \m E$.
\end{itemize}
\end{Defi}

We can now state the theorem which is our goal.
\begin{thm}\label{thm:factormonoidal}
Let $C$ be a category equipped with a pair of monoidal structures $(C, \otimes_{1}, I_{1})$ and $(C, \otimes_{2}, I_{2})$, and further that the identity functor $1: C \rightarrow C$ is given the structure of an oplax monoidal functor $(C, \otimes_{1}, I_{1}) \rightarrow (C, \otimes_{2}, I_{2})$ with structure maps
\[
\begin{array}{c}
f_{x,y}:x \otimes_{1} y \rightarrow x \otimes_{2} y, \\
f_{0}:I_{1} \rightarrow I_{2}.
\end{array}
\]
Assume $C$ has an orthogonal factorization system $(\m E,\m M)$ which satisfies the following properties:
\begin{itemize}
\item if $f, f' \in \m E$, then $f \otimes_{1} f' \in \m E$, and
\item if $g, g' \in \m M$, then $g \otimes_{2} g' \in \m M$.
\end{itemize}
Then factoring the $f_{x,y}, f_{0}$ as an $\m E$-map followed by an $\m M$-map
\[
\xy
(0,0)*+{x \otimes_{1} y}="00";
(25,0)*+{x \otimes_{2} y}="20";
(12.5,-10)*+{x \otimes y}="10";
(50,0)*+{I_{1}}="40";
(62.5, -10)*+{I}="50";
(75,0)*+{I_{2}}="60";
{\ar^{f_{x,y}} "00";  "20"};
{\ar_{e_{x,y}} "00"; "10"};
{\ar_{m_{x,y}} "10"; "20" };
{\ar^{f_{0}} "40"; "60" };
{\ar_{e_{0}} "40"; "50" };
{\ar_{m_{0}} "50"; "60" };
\endxy
\]
gives a functor $\otimes:C^{2} \rightarrow C$, an object $I \in C$, and natural transformations
\[
\otimes_{1} \stackrel{e}{\Longrightarrow} \otimes \stackrel{m}{\Longrightarrow} \otimes_{2}.
\]
Furthermore, $C$ admits a unique monoidal structure with tensor product $\otimes$ and unit object $I$ such that $e,m$ each give oplax monoidal structures to the identity functor.  If the monoidal structures $\otimes_{i}$ are additionally equipped with symmetric (resp., braided) structures, and the $f_{x,y}, f_{0}$ give the identity functor the structure of an oplax symmetric (resp., braided) monoidal functor, then we can additionally give $(C, \otimes, I)$ a unique symmetry (resp., braid) such that $e, m$ each give oplax symmetric (resp., braided) monoidal structures to the identity functor.
\end{thm}

This result will allow us to construct the symmetric monoidal structure on \twocat using the Gray tensor product, once we know it fits into such a factorisation.  Now it is not hard to convince oneself that the above result is true, but to verify the coherence axioms for the factored monoidal structure directly is a significant undertaking.  In Theorem~\ref{thm:LawvereTheory} we describe a more general result on factorisations in 2-dimensional universal algebra that will give the above result as a special case.  It turns out that the abstract setting best suited to proving such results is that of Lawvere 2-theories rather than 2-monads.  Since the subject is less well developed we describe the necessary background below.

\subsection{Lawvere theories}

\begin{Defi}
Let $\g F$ denote the skeleton of the category of finite sets with objects $[n] = \{ 1, 2, \ldots, n \}$.
\end{Defi}
By a finite power of an object $a$ we mean a product $a^{n}$ where $n$ is finite.  Dually finite copowers are certain finite coproducts.  Observe that the category $\g F$ has a canonical choice of finite copowers, and so $\g F^{op}$ has a canonical choice of finite powers.

\begin{Defi}
A \emph{Lawvere theory} \cite{Lawvere1963Functorial} is a category $\g T$ equipped with an identity on objects functor $\g F^{op} \to \g T$ which preserves finite powers of $1$.
\end{Defi}

\begin{rem}
It follows that a Lawvere theory $\g T$ comes equipped with a choice of finite powers of $1$ and, indeed, of finite products.
\end{rem}

In the standard definition of model, one considers models of $\g T$ in a category $C$ with finite products.  A model then consists of a finite product preserving functor $\g T \to C$ and the category of models consists of the full subcategory $FP(\g T,C) \to [\g T,C]$ of product preserving functors.  The first basic fact about Lawvere theories is that categories of the form $FP(\g T,\Set)$ coincide, up to equivalence of categories over $\Set$, with the categories of models of equational theories $\textnormal{T}=(\Omega,E)$.  Here $\Omega = \{\Omega(n):n \in \g N\}$ is a signature and $E$ a set of equations.  For a textbook treatment of this result see Propositions 3.2.9 and 3.3.4 of \cite{Borceux1994Handbook2}.

Now we would prefer our categories (later 2-categories) of models to capture the algebraic categories \emph{up to isomorphism} rather than equivalence, primarily because the result that we will prove, \cref{thm:factormonoidal}, refers to structures such as identity morphisms that are not invariant under equivalence.  Accordingly we use the following definition.

\begin{Defi}\label{Defi:model}
A \emph{model} of a Lawvere theory $\g T$ in a category $C$ equipped with chosen finite powers is a functor $A:\g T \to C$ strictly preserving finite powers of $1$.  These are the objects of the category $Mod(\g T,C)$ whose morphisms are natural transformations.
\end{Defi}

\begin{rem}
Let us  compare the two notions.  If $C$ has both specified finite powers and finite products -- as in all of our examples -- then we have an inclusion $Mod(\g T,C) \to FP(\g T,C)$ over $C$ as below,
\begin{equation*}
\cd{Mod(\g T,C)\ar[dr]_{U} \ar[rr] && FP(\g T,C) \ar[dl]^{V}\\
& C}
\end{equation*}
and the inclusion is an equivalence of categories.  Here both forgetful functors $U$ and $V$ are given by evaluation at $1$.  Now the key property of $U$, not shared by $V$, is that each isomorphism $f:X \to UY$ lifts along $U$ to a unique isomorphism $f^{\star}:X^{\star} \to Y$ over $f$.  A functor having this property is said to be a \emph{discrete isofibration} (or a \emph{uniquely transportable} functor \cite{Adamek2004The-Joy}.)

As mentioned above, a Lawvere theory $\g T$ corresponds to an equational theory $T=(\Omega,E)$ for which there is an equivalence $FP(\g T,\Set) \simeq Mod(\textnormal{T},\Set)$ over $\Set$ as in the right triangle below.
\begin{equation*}
\cd{Mod(\g T,\Set)\ar[dr]_{U} \ar[r] & FP(\g T,\Set) \ar[d]_{V} \ar[r] & Mod(\textnormal{T},\Set) \ar[dl]^{W}\\
& \Set}
\end{equation*}
We then obtain a composite equivalence $Mod(\g T,\Set) \simeq Mod(\textnormal{T},\Set)$ over $\Set$.  Because both $U$ and $W$ are discrete isofibrations the composite equivalence must be an isomorphism.  Arguing in this way, we conclude that \emph{categories of the form $Mod(\g T,\Set)$ coincide, up to concrete isomorphism, with the categories of models of equational theories.}
\end{rem}
\begin{rem}
In the case of an equational theory $(\Omega,E)$ whose set $E$ of equations is empty, the corresponding Lawvere theory $F\Omega$ is called the free Lawvere theory on the signature $\Omega$.  It owes its name to the fact that it arises from an adjunction between the category of Lawvere theories and of signatures, but we will need the following syntactic definition.
\end{rem}
\begin{Defi}\label{Defi:free_theory}
Let $\Omega$ be a signature. The \emph{free Lawvere theory} $F\Omega$ has $F\Omega(n,m)=F\Omega(n,1)^{m}$ where $F\Omega(n,1)=T_{\Omega}(n)$ is the set of \emph{$\Omega$-terms in $n$ variables}.  The sets of $\Omega$-terms are inductively defined by
\begin{itemize}
\item
for each $i \in \{1,\ldots ,n\}$ there is a given variable $x_{i} \in T_{\Omega}(n)$, and
\item
given $t_{1} \in T_{\Omega}(m_{1}), \ldots ,t_{n} \in T_{\Omega}(m_{n})$ and $f \in \Omega(n)$ then $f(t_{1} \ldots t_{n}) \in T_{\Omega}(m_{1}+\ldots + m_{n})$.
\end{itemize}
Composition in $F\Omega$ is given by substitution of terms.
\end{Defi}

\begin{conv}
We often identify $f \in \Omega(n)$ with the $\Omega$-term $f(x_{1}, \ldots, x_{n})$.
\end{conv}

We have an isomorphism $Mod(F\Omega,C) \cong Mod(\Omega,C)$ where this last category consists of \emph{$\Omega$-algebras}: objects $a \in C$ equipped with a map $f^{a}:a^{n} \to a$ for each $f \in \Omega(n)$.  From an $F\Omega$-model $A$ the corresponding algebra is simply $A(1)$ equipped with the action
\begin{equation*}
\cd{A(1)^{n} \ar[rrr]^{A(f)} &&& A(1).}
\end{equation*}
From an $\Omega$-algebra $a$ the corresponding $F\Omega$-model sends  $x_{i}:n \to 1$ to the $i$'th product projection $a(x_{i}):a^{n} \to a$, and is inductively defined on $f(t_{1}, \ldots, t_{n})$ as
\begin{equation}\label{eq:Omega1}
\cd{a^{m_{1}+\ldots +m_{n}} \ar[rrr]^{a(t_{1}) \times \cdots \times a(t_{n})}&&&  a^{n} \ar[rr]^{a(f)} && a.}
\end{equation}

\subsection{From theories to 2-theories}
There are several possible generalisations of Lawvere theories to enriched category theory and 2-category theory, some of which are covered in \cite{Gray19732-algebraic, Power1999Enriched, Yanofsky2000The-syntax, Yanofsky2001Coherence, Power2005Discrete, Cohen2009Thesis, Lack2011Notions}.  We are only interested in the most basic generalisation, which appears to have been first mentioned in print in \cite{Gray19732-algebraic}, and has been further studied in \cite{Yanofsky2000The-syntax, Yanofsky2001Coherence, Cohen2009Thesis}.

\begin{conv}
When we speak of powers in a 2-category below we are referring to certain products, but now in the stronger enriched sense: the product $\Pi_{i \in I} a_{i}$ in a 2-category $C$ is characterised by a natural isomorphism of \emph{categories}
\[
C(x, \Pi_{i \in I} a_{i}) \cong \Pi_{i \in I} C(x,a_{i})
\]
rather than sets.
\end{conv}

In the following definition the category $\g F$ is viewed as a locally discrete 2-category: one in which each 2-cell is an identity.

\begin{Defi}
A \emph{Lawvere 2-theory} is a 2-category $\g T$ equipped with an identity on objects 2-functor $\g F^{op} \to \g T$ which preserves finite powers of $1$.
\end{Defi}
Such equips $\g T$ with a canonical choice of finite powers of $1$.  Again, we will only consider models of 2-theories in 2-categories $C$ equipped with a \emph{specified choice} of finite powers $a^{n}$ of each object.

\begin{Defi}
A \emph{model} of a Lawvere 2-theory $\g T$ in $C$ is a 2-functor $A:\g T \to C$ strictly preserving finite powers of $1$.
\end{Defi}
We will now discuss the various notions of morphism between models of a Lawvere 2-theory $\g T$.  There are, just as with algebras over a 2-monad, notions of strict, strong, lax and oplax morphisms of models.

\begin{Defi}
A \emph{strict morphism} $f:A \to B$ of $\g T$-models is a 2-natural transformation $A \Rightarrow B:\g T \to C$.
\end{Defi}

Since our focus will be on oplax morphisms, we give that definition, with the definition of a lax or strong morphism the obvious analogue in which the oplax transformation is replaced by a lax or pseudonatural one, respectively.  Before doing so, we recall the definition of an oplax transformation, stated here in the context of 2-functors between 2-categories but easily adapted to weaker contexts.

\begin{Defi}\label{Defi:oplax_trans}
Let $A,B$ be 2-categories, and $F,G: A \to B$ 2-functors.  An \emph{oplax transformation} $\alpha: F \Rightarrow G$ consists of
\begin{itemize}
\item 1-cells $\alpha_a : Fa \to Ga$ for all objects $a \in A$ and
\item 2-cells $\alpha_f: \alpha_b Ff \Rightarrow Gf \alpha_a$ for all 1-cells $f: a \to b$ in $A$
\end{itemize}
such that the following diagrams commute for all objects $a$, 2-cells $\tau: f \Rightarrow f'$, and composable pairs $a \stackrel{f}{\to} b \stackrel{g}{\to} c$, respectively, in $A$.
\[
\xy
(0,0)*+{\alpha_a}="11";
(25,0)*+{\alpha_a \, F\id}="12";
(25,-10)*+{G\id \, \alpha_a}="13";
{\ar@{=} "11"; "12"};
{\ar@{=>}^{\alpha_{\id}} "12"; "13"};
{\ar@{=} "11"; "13"};
(50,0)*+{\alpha_b  \, Ff}="31";
(75,0)*+{Gf  \, \alpha_a}="32";
(50,-10)*+{\alpha_b  \, Ff'}="33";
(75,-10)*+{Gf'  \, \alpha_a}="34";
{\ar@{=>}^{\alpha_f} "31"; "32"};
{\ar@{=>}^{G\tau * 1} "32"; "34"};
{\ar@{=>}_{1* F\tau} "31"; "33"};
{\ar@{=>}_{\alpha_{f'}} "33"; "34"};
(7.5,-20)*+{\alpha_c  \, Fg  \, Ff}="21";
(37.5,-20)*+{Gg  \, \alpha_b  \, Ff}="22";
(67.5,-20)*+{Gg  \, Gf  \, \alpha_a}="23";
(7.5,-30)*+{\alpha_c  \, F(gf)}="24";
(67.5,-30)*+{G(gf)  \, \alpha_a}="25";
{\ar@{=>}^{\alpha_g * 1} "21"; "22"};
{\ar@{=>}^{1* \alpha_f} "22"; "23"};
{\ar@{=} "23"; "25"};
{\ar@{=>}_{\alpha_{gf}} "24"; "25"};
{\ar@{=} "21"; "24"};
\endxy
\]
\end{Defi}

We learned of the following definition from \cite{LackTalk2007}.  An equivalent, but more complicated, definition was described earlier in \cite{Yanofsky2000The-syntax}.
\begin{Defi}\label{Defi:oplax2}
An \emph{oplax morphism} $f:A \rightsquigarrow B$ of $T$-models is an oplax natural transformation $A \Rightarrow B$ such that the composite
\begin{equation*}
\xy
(0,0)*+{\g F^{op}}="00";(25,0)*+{\g T}="10";(50,0)*+{C}="20";
{\ar^{} "00"; "10"};
{\ar@/^3ex/^{A} "10"; "20"}; {\ar@/_3ex/_{B} "10"; "20"};
{\ar@{=>}^{f}(37,4)*+{};(37,-4)*+{}};
\endxy
\end{equation*}
is 2-natural.
\end{Defi}

This last restriction forces the component $f(n):A(n) \to B(n)$ of an oplax morphism $f$ to equal $f(1)^{n}:A(1)^{n} \to B(1)^{n}$, and ensures that the oplax transformation is determined by 2-cell components as below.
\begin{equation*}
\xy
(0,0)*+{A(1)^{n}}="00";(25,0)*+{B(1)^{n}}="10";
(0,-20)*+{A(1)}="01";(25,-20)*+{B(1)}="11";
{\ar^{f(1)^{n}} "00"; "10"};{\ar_{f(1)} "01"; "11"};
{\ar_{A(t)} "00"; "01"};{\ar^{B(t)} "10"; "11"};
{\ar@{=>}^{f(t)}(8,-12)*+{};(16,-8)*+{}};
\endxy
\end{equation*}
This definition fits with what one expects from the definition of an oplax monoidal functor.  There are 2-cells between strict or oplax morphisms of models, the modifications, but these will not concern us.  Accordingly one has 2-categories $Mod(T,C)_{s}$ and $Mod(T,C)_{o}$ of strict and oplax morphisms of models.

\begin{conv}
If a $T$-model $X:\g T \to C$ has $X(1)=a$ then we call $X$ a $T$-algebra structure \emph{on $a$} and if an oplax morphism $f:X \to Y$ between two different $T$-algebra structures on $a$ has $f(1)=1:a \to a$ then we say that $f$ is an \emph{oplax morphism over the identity on $a$}.  This is equally to say that the oplax natural transformation $f$ is an \emph{icon} (see \cref{Defi:icon} below).
\end{conv}

\begin{Defi}\label{Defi:icon}
Let $F,G: A \to B$ be 2-functors, and assume that $Fa = Ga$ for all objects $a \in A$.  An \emph{icon} \cite{Lack2010Icons} $\alpha$ is an oplax transformation $\alpha:F \Rightarrow G$ such that $\alpha_a$ is the identity morphism on $Fa = Ga$ for all objects $a \in A$.  This amounts to giving, for each 1-cell $f: a \to b$, a 2-cell $\alpha_f:Ff \Rightarrow Gf$ satisfying three axioms as for oplax transformations.
\end{Defi}

\begin{rem}
The abstract tool most commonly used to model algebraic structures borne by categories is that of a 2-monad, rather than a 2-theory.  The special class of \emph{strongly finitary 2-monads} on \Cat \cite{Kelly1993Finite-product-preserving} capture monoidal categories and their symmetric variants, and more generally those structures whose operations $C^{n} \to C$ involve only functors from a finite power of $C$ to itself.  In order to show that 2-theories model such categorical structures it consequently suffices to establish their correspondence with strongly finitary 2-monads.  We outline this correspondence now.  For further details, in a general context, see Section 6 of \cite{Lack2011Notions}.

From such a 2-monad $T$ the corresponding 2-theory $\g T$ has $\g T(n,m)=\Cat(m,Tn)$ with composition inherited from the Kleisli 2-category $\Cat_{T}$.  The composite inclusion $j:\g T^{op} \to \Cat_{T} \to \TAlgs$ to the 2-category of strict $T$-algebras and strict morphisms induces a singular 2-functor $\tilde{j}:\TAlgs \to [\g T,\Cat]$ that restricts to an equivalence $\tilde{j}:\TAlgs \simeq FP(\g T,\Cat)_{s}$ with the 2-category of finite product preserving functors and 2-natural transformations.  This is a special case of Theorem 6.6 of \cite{Lack2011Notions}.  Composing this with the equivalence $FP(\g T,\Cat)_{s} \simeq Mod(\g T,\Cat)_{s}$ gives rise to an equivalence $\TAlgs \simeq Mod(\g T,\Cat)_{s}$ over $\Cat$ up to 2-natural isomorphism.  Since both forgetful 2-functors $\TAlgs \to \Cat$ and $ Mod(\g T,\Cat)_{s} \to \Cat$ have discrete isofibrations for underlying functors, the equivalence gives rise to an isomorphism $$\TAlgs \cong Mod(\g T,\Cat)_{s}$$ over $\Cat$.

It is an as-yet unpublished result of Lack and Power \cite{LackPower} that this correspondence extends to match weak $\g T$-algebra morphisms (using the 2-theory) with the corresponding weak $T$-algebra maps (using the 2-monad).  In particular one obtains an isomorphism of 2-categories $\TAlgo \cong Mod(\g T,\Cat)_{o}$ over \Cat.  We note that this last result can be deduced by an application of Theorem 25 of \cite{Bourke2014Two-dimensional}.
\end{rem}

The above is quite abstract.  It might be helpful to point out that in many cases  the 2-theory $\g T$ can be described quite explicitly.

\begin{example}
Let $\g T$ be the theory for monoidal categories.  We have
\[
\g T(n,1)=\Cat(1,Tn)\cong Tn,
\]
 the free monoidal category on $n$ elements.  This has objects the elements of the free pointed magma on $n$.  These are bracketed words in $\{x_{1}, \ldots, x_{n},e\}$ such as $(x_{1}(x_{2}x_{4}))e$.  A unique isomorphism connects two such elements in $\g T(n,1)$ just when they are identified as words in the free monoid on $n$ elements, with $e$ being sent to the identity element.  Of course $\g T(n,m)=\g T(n,1)^{m}$ and composition is by substitution.
\end{example}

\begin{nota}
A Lawvere theory $\g T$ can be viewed as a 2-theory in which each 2-cell is an identity, and we use the same name to indicate the 2-theory.
\end{nota}

In the case of a Lawvere theory viewed as a 2-theory, a $\g T$-model in a 2-category $ C$ just amounts to a $\g T$-model in the underlying category $ C_{0}$.  Something new only appears when considering weak $\g T$-morphisms  and the case of interest is, again, the free Lawvere theory $F\Omega$ on a signature.  In that case we have an isomorphism of 2-categories
\[
Mod(F\Omega, C)_{o} \cong Mod(\Omega,C)_{o}.
\]
Here a morphism $f:a \to b$ of $Mod(\Omega,C)_{o}$ is an oplax morphism of $\Omega$-algebras: that is, a morphism $k:a \to b$ together with a 2-cell $k(f):k \circ f^{a} \Rightarrow f^{b} \circ k^{n}$ for each $f \in \Omega(n)$.  The 2-cells of $Mod(\Omega,C)_{o}$ are the evident ones.

Given an oplax morphism $k:A \to B$ of $F\Omega$-algebras the oplax $\Omega$-algebra map is $k(1):A(1) \to B(1)$ equipped with the 2-cell $k(f):k(1) \circ A(f) \Rightarrow B(f) \circ k(1)^{n}$ at $f \in \Omega(n)$.  From an oplax morphism $k:a \to b$ of $\Omega$-algebras the corresponding oplax $F\Omega$-algebra map is defined inductively at $f(t_{1}, \ldots t_{n}) \in F\Omega(n,1)$ by the pasting below.
\begin{equation}\label{eq:Omega2}
\xy
(0,0)*+{a^{m_{1}+\ldots +m_{n}}}="00";
(50,0)*+{a^{n}}="10";
(90,0)*+{a}="20";
(0,-20)*+{b^{m_{1}+\ldots +m_{n}}}="01";
(50,-20)*+{b^{n}}="11";
(90,-20)*+{b}="21";
{\ar^{t^{a}_{1} \times \ldots \times t^{a}_{n}} "00"; "10"};{\ar^{f^{a}} "10"; "20"};
{\ar_{t^{b}_{1} \times \ldots \times t^{b}_{n}} "01"; "11"};{\ar_{f^{b}} "11"; "21"};
{\ar_{k^{m_{1}+\ldots +m_{n}}} "00"; "01"};{\ar_{k^{n}} "10"; "11"};{\ar^{k} "20"; "21"};
{\ar@{=>}^{k(t_{1})\times \ldots \times k(t_{n})}(18,-7)*+{};(18,-13)*+{}};{\ar@{=>}^{k(f)}(70,-7)*+{};(70,-13)*+{}};
\endxy
\end{equation}

\begin{Defi}
Let $\g T$ be a Lawvere 2-theory.  Then the underlying category $\g T_{0}$ is a Lawvere theory, called the \emph{underlying Lawvere theory}.
\end{Defi}

The categorical structures of interest to us -- monoidal categories and their symmetric variants -- have an appealing property: they involve, in their definitions, no equations between objects.  For example, the definition of a monoidal category involves isomorphisms $a \otimes (b \otimes c) \cong (a \otimes b) \otimes c$ between objects but not an equality.  This absence of equations between objects is elegantly captured in terms of the underlying theory of a 2-theory: for $\g T$ the 2-theory for monoidal categories the theory $\g T_{0}$ is free on a signature, namely the signature $\Omega$ with $\Omega(0)=1, \Omega(2)=1$ and $\Omega(n)=\emptyset$ otherwise, whose algebras are \emph{pointed magmas}.  The same theory $F\Omega$ underlies the 2-theories for braided and symmetric monoidal categories.  We note the general case now.

\begin{thm}\label{thm:no_eq_on_obj}
A 2-theory $\g T$ presents a categorical structure involving no equations on objects if and only if its underlying theory is free on a signature.
\end{thm}
\begin{proof}
The corresponding result for strongly finitary 2-monads was established in Section 6 of \cite{Bourke2011On-semiflexible}.  In that paper the condition of a 2-monad presenting a categorical structure involving no equations between objects was made precise, and called \emph{pie-presentability} because of the relationship with pie colimits.  Taken together, Theorem 34 and Proposition 36 of \cite{Bourke2011On-semiflexible} assert that a strongly finitary 2-monad on $\Cat$ is pie-presentable just when its underlying monad on $\Set$ is \emph{free on a signature}.

Transporting this monad theoretic result across the equivalence between strongly finitary 2-monads on $\Cat$ and Lawvere 2-theories yields the present result.
\end{proof}

In the case that $\g T_{0}$ is free on a signature $\Omega$ we obtain an evident forgetful 2-functor
\[
U:Mod(\g T)_{o} \to Mod(F\Omega)_{o} \cong Mod(\Omega)_{o}
\]
obtained by restricting $\g T$-models and their oplax morphisms along $F \Omega \hookrightarrow \g T$.

\begin{thm}\label{thm:LawvereTheory}
Let $\g T$ be a Lawvere 2-theory whose underlying Lawvere theory $\g T_{0}$ is free on the signature $\Omega$, and let $U:Mod(\g T)_{o} \to Mod(\Omega)_{o}$ denote the forgetful 2-functor to $\Omega$-algebras.  Consider $\g T$-algebra structures $X$ and $Y$ on a category $A$ and an oplax morphism $k:X \rightsquigarrow Y$ over the identity on $A$.  Suppose that $A$ has a factorisation system $(\m E, \m M)$, and further that for each morphism $f \in \Omega(n) \subseteq \g T(n,1)$ the functor $X(f):A^{n} \to  A$ preserves pointwise $\m E$'s and $Y(f): A^{n} \to  A$ preserves pointwise $\m M$'s.

For each $f \in \Omega(n)$ factor the natural transformation $k(f):X(f) \Rightarrow Y(f) \in [ A^{n}, A]$ into
\[
\cd{X(f) \ar@{=>}[r]^{e(f)} & Z(f) \ar@{=>}[r]^{m(f)} & Y(f)}
\]
as pointwise $\m E$ followed by pointwise $\m M$.  The $\Omega$-algebra $Z$ and oplax $\Omega$-algebra maps $e$ and $m$ lift uniquely along $U$ to $\g T$-algebra structures on $A$ and oplax $\g T$-algebra morphisms over the identity on $A$.
\end{thm}
\begin{proof}
The $\Omega$-algebra structure $Z$ extends uniquely to an $F\Omega$-algebra on $A$.  As described in \eqref{eq:Omega1} this has value at a term $f(t_{1}, \ldots, t_{n}) \in F\Omega(n,1)$ given inductively by
\begin{equation*}
\cd{A^{m_{1}+\ldots +m_{n}} \ar[rrr]^{Z(t_{1}) \times \ldots \times Z(t_{n})}&&&  A^{n} \ar[rr]^{Z(f)} && A.}
\end{equation*}
Likewise the oplax morphisms of $\Omega$-algebras $e$ and $m$ extend uniquely to oplax morphisms of $F\Omega$-algebras $e:X \to Z$ and $m:Z \to Y$ over $ A$ according to \eqref{eq:Omega2}.

We will show that for all $t \in F\Omega(n,m)$ the natural transformations
\[
e(t):X(t) \Rightarrow Z(t) \in [A^{n},A^{m}], \quad m(t):Z(t) \Rightarrow Y(t) \in [A^{n},A^{m}]
\]
are pointwise $\m E$ and pointwise $\m M$, respectively.  Let us consider first $e$.  It clearly suffices to consider the case $m=1$ and for $t \in F\Omega(n,1)$ we can argue inductively. At a variable $x_{i} \in F\Omega(n,1)$ the natural transformation $e(x_{i}):X(x_{i}) \Rightarrow Z(x_{i}) \in [ A^{n}, A]$ is the identity on the $i$'th product projection.  This is certainly pointwise $\m E$.  At a term $f(t_{1}, \ldots, t_{n})$ the natural transformation
\[
e(f(t_{1},\ldots t_{n})):X(f(t_{1},\ldots t_{n})) \Rightarrow Z(f(t_{1},\ldots ,t_{n}))
\]
is defined as the composite below.
\begin{equation*}
\xy
(0,0)*+{ A^{m_{1} + \ldots m_{n}}}="00";(50,0)*+{ A^{n}}="10";(100,0)*+{ A}="20";
{\ar@/^3.5ex/^{X(t_{1}) \times \cdots \times  X(t_{n})} "00"; "10"}; {\ar@/_3.5ex/_{Z(t_{1}) \times \cdots \times Z(t_{n})} "00"; "10"};
{\ar@/^3ex/^{X(f)} "10"; "20"}; {\ar@/_3ex/_{Z(f)} "10"; "20"};
{\ar@{=>}|{e(t_{1}) \times \cdots  \times e(t_{n})}(25,5)*+{};(25,-5)*+{}};
{\ar@{=>}^{e(f)}(75,4)*+{};(75,-4)*+{}};
\endxy
\end{equation*}
Observe that if each $e(t_{i})$ is pointwise $\m E$ then so is their product.  By assumption $X(f)$ preserves pointwise $\m E$'s so that the whiskering $X(f) \circ (e(t_{1}) \times \cdots \times e(t_{n}))$ is pointwise $E$, whilst the whiskering $e(f) \circ (Z(t_{1}) \times \cdots \times Z(t_{n}))$ is trivially pointwise $E$.  That  $e(f(t_{1},\ldots t_{n}))$, the composite of these two, is pointwise $\m E$ now follows from the fact that the class $\m E$ is closed under composition.   Similarly the oplax morphism $m:Z \to Y$ of $F\Omega$-algebras is pointwise $\m M$ in each component.

It remains to describe the action of $Z$ on 2-cells of $\g T$.  So consider a 2-cell $\theta:s \Rightarrow t$.  In order for the oplax $F\Omega$-algebra maps $e$ and $m$ to be oplax $\g T$-algebra maps we require exactly that the left and right squares below respectively commute.
\[
\xy
(0,0)*+{X(s)}="00"; (20,0)*+{Z(s)}="10";(40,0)*+{Y(s)}="20";
(0,-20)*+{X(l)}="01";(20,-20)*+{Z(l)}="11";(40,-20)*+{Y(l)}="21";
{\ar@{=>}^{e(s)} "00"; "10"};
{\ar@{=>}_{X(\theta)} "00"; "01"};
{\ar@{=>}|{Z(\theta)} "10"; "11"};
{\ar@{=>}_{e(t)} "01"; "11"};
{\ar@{=>}^{m(s)} "10"; "20"};
{\ar@{=>}_{m(t)} "11"; "21"};
{\ar@{=>}^{Y(\theta)} "20"; "21"};
\endxy
\]
This forces us to define $Z(\theta)$ as the unique filler, whose existence is guaranteed by the fact that  $e(s)$ is pointwise $\m E$ and $m(s)$ pointwise $\m M$.  That $Z$ preserves vertical composition of 2-cells follows easily from the uniqueness of diagonal fillers.  Given horizontally composable 2-cells $\theta:s \Rightarrow t$ and $\theta^{\prime}:s^{\prime} \Rightarrow t^{\prime}$ consider the following diagrams.
\[
\xy
(0,0)*+{Xs^{\prime}Xs}="00"; (26,0)*+{Zs^{\prime}Zs}="10";(52,0)*+{Ys^{\prime}Ys}="20";
(0,-20)*+{Xt^{\prime}Xt}="01";(26,-20)*+{Zt^{\prime}Zt}="11";(52,-20)*+{Yt^{\prime}Yt}="21";
{\ar@{=>}^{e(s^{\prime})e(s)} "00"; "10"};
{\ar@{=>}|{X(\theta^{\prime})X(\theta)} "00"; "01"};
{\ar@{=>}|{Z(\theta^{\prime})Z(\theta)} "10"; "11"};
{\ar@{=>}_{e(t^{\prime})e(t)} "01"; "11"};
{\ar@{=>}^{m(s^{\prime})m(s)} "10"; "20"};
{\ar@{=>}_{m(t^{\prime})m(t)} "11"; "21"};
{\ar@{=>}|{Y(\theta^{\prime})Y(\theta)} "20"; "21"};
\endxy
\hspace{0.45cm}
\xy
(0,0)*+{X(s^{\prime}s)}="00"; (22,0)*+{Z(s^{\prime}s)}="10";(44,0)*+{Y(s^{\prime}s)}="20";
(0,-20)*+{X(t^{\prime}t)}="01";(22,-20)*+{Z(t^{\prime}t)}="11";(44,-20)*+{Y(t^{\prime}t)}="21";
{\ar@{=>}^{e(s^{\prime}s)} "00"; "10"};
{\ar@{=>}|{X(\theta^{\prime}\theta)} "00"; "01"};
{\ar@{=>}|{Z(\theta^{\prime}\theta)} "10"; "11"};
{\ar@{=>}_{e(t^{\prime}t)} "01"; "11"};
{\ar@{=>}^{m(s^{\prime}s)} "10"; "20"};
{\ar@{=>}_{m(t^{\prime}t)} "11"; "21"};
{\ar@{=>}|{Y(\theta^{\prime}\theta)} "20"; "21"};
\endxy
\]
All components on the outsides agree by functoriality of $X$ and $Y$ on 1-cells and 2-cells, $Z$ on 1-cells and using that $e$ and $m$ are oplax $F\Omega$-algebra maps on $ A$; therefore by the uniqueness of diagonal fillers the diagonals also agree.  Therefore $Z$ is a 2-functor and strictly preserves finite powers of $1$ because its underlying functor, an $F\Omega$-algebra, does so.

\end{proof}

\begin{proof}[Proof of Theorem~\ref{thm:factormonoidal}]
Let $\g T$ be the 2-theory for monoidal categories.  $\g T_{0}$ is free on the signature $\Omega$ for pointed magmas.  We have an isomorphism of 2-categories $MonCat_{o} \cong Mod(\g T,\Cat)_{o}$ which commutes with the respective forgetful 2-functors to $Mod(\Omega,\Cat)_{o}$.  The part of Theorem~\ref{thm:factormonoidal} concerning monoidal structure concerns the forgetful 2-functor $MonCat_{o} \to Mod(\Omega,\Cat)_{o}$, but this corresponds exactly to Theorem~\ref{thm:LawvereTheory}.  For the parts concerning symmetric or braided structures, we replace $\g T$ by the appropriate 2-theory.
\end{proof}

\section{The cartesian, funny and Gray tensor products}
We are interested in studying monoidal structures on \twocat, the category of small 2-categories and 2-functors.  First, it is a cartesian closed category: given 2-categories $A$ and $B$ we denote their cartesian product by $A \times B$ and we write $[A,B]$ for the corresponding internal hom.
\begin{Defi}
The 2-category $[A,B]$ has objects 2-functors $F:A \to B$, 1-cells 2-natural transformations between them and 2-cells modifications between those.
\end{Defi}

These 2-natural transformations are the strictest possible kind of transformation between 2-functors.  On the other end of the spectrum, we can consider the notion of transformation satisfying the fewest axioms.

\begin{Defi}
Let $F,G:A \to B$ be 2-functors.  A \emph{transformation} $\alpha:F \Rightarrow G$ consists of a 1-cell $\alpha_{x}:Fx \to Gx$ for each object $x \in A$, subject to no axioms.
\end{Defi}

Transformations are the 1-cells in another closed structure on \twocat.

\begin{Defi}
The 2-category $[A,B]_{f}$ has objects 2-functors $F:A \to B$ and hom-categories given by
\[
[A,B]_{f}(F,G)=\Pi_{x\in A}B(Fx,Gx).
\]
\end{Defi}

If transformations are the 1-cells in $[A,B]_{f}$, then the 2-cells between them are once again given by a family of 2-cells
\[
\Gamma_{x}:\alpha_{x} \Rightarrow \beta_{x},
\]
indexed by the objects of $A$, and once again subject to no axioms.

At a 2-category $D$ let $obD$ be the discrete 2-category with the same set of objects as $D$ and let $i:obD \to D$ be the evident inclusion.  Observe that to give $A \to [B,C]_{f}$ is equally to give a pair of maps $obA \to [B,C]$ and $A \to [obB,C]$ rendering commutative the diagram left below
\[
\cd{obA \ar[r] \ar[d]_{i} & [B,C] \ar[d]^{i^*{}} & & obA \times ob B \ar[r]^{1 \times i} \ar[d]_{i \times 1} & obA \times B \ar[d]\\
A \ar[r] & [obB,C] && A \times ob B \ar[r] & C}
\]
and that, by adjointness, these in turn correspond to a pair of maps $obA \times B \to C$ and $A \times obB \to C$ rendering commutative the right diagram.

\begin{Defi}\label{Defi:funny}
The \emph{funny tensor product} $A \star B$ is the pushout
\[
\cd{obA \times ob B \ar[r]^{1 \times i} \ar[d]_{i \times 1} & obA \times B \ar[d]\\
A \times ob B \ar[r] & A \star B \hspace{0.2cm} .}
\]
\end{Defi}

By the preceding discussion the funny tensor product $A \star B$ is characterised by a natural isomorphism
\[
\twocat(A \star B,C) \cong \twocat(A,[B,C]_{f}).
\]
Using the pushout description, it is a simple exercise using universal properties to show that the funny tensor product is symmetric monoidal with unit the terminal 2-category.

At 2-categories $A$ and $B$ we have the inclusion
\begin{equation}\label{eq:j}
J_{A,B}:[A,B] \to [A,B]_{f}
\end{equation}
that forgets naturality.  Because these inclusions are natural in both $A$ and $B$, and because the vertical arrows in the diagram below are isomorphisms,
\[
\xy
(0,0)*+{\twocat(A,[B,C])}="00"; (50,0)*+{\twocat(A,[B,C]_{f})}="10";
(0,-15)*+{\twocat(A \times B,C)}="01"; (50,-15)*+{\twocat(A \star B,C)}="11";
{\ar^{\twocat(A,J)} "00";"10"};
{\ar ^{}"10";"11"};
{\ar _{}"00";"01"};
{\ar@{.>} _{\twocat(K,C)} "01";"11"};
\endxy
\]
the Yoneda lemma induces a unique $K:A \star B \to A \times B$, natural in 2-functors in both variables, such that the above diagram commutes.  Another description of $K:A \star B \to A \times B$ is as the unique map from the pushout $A \star B$ corresponding to the commutative square below.
\[
\xy
(0,0)*+{obA \times obB}="00"; (25,0)*+{obA \times B}="10";
(0,-15)*+{A \times obB}="01"; (25,-15)*+{A\times B}="11";
{\ar^{} "00";"10"};
{\ar ^{}"10";"11"};
{\ar _{}"00";"01"};
{\ar _{} "01";"11"};
\endxy
\]
Using this latter description it is easily seen that the comparison 2-functors $K$ yield the structure
\[
(1,K):(\twocat,\star) \to (\twocat,\times)
\]
of a symmetric oplax structure on the identity functor $1:\twocat \to \twocat$ (see \cref{rem:oplax-id}).

\begin{Defi}\label{Defi:pseudo_trans}
A \emph{pseudonatural transformation} is an oplax transformation $\alpha$ such that the 2-cells $\alpha_{f}$ are invertible for each 1-cell $f$.
\end{Defi}

Between pseudonatural transformations $\eta$ and $\mu$ are modifications; such are specified by a family of 2-cells $\Gamma_{x}:\eta_{x} \Rightarrow \mu_{x}$ compatible with $\eta$ and $\mu$.  For the details we refer the reader to  \cite{Benabou1967Introduction}.

\begin{Defi}
The 2-category $Ps(A,B)$ is defined to have objects the 2-functors $F:A \to B$, 1-cells the pseudonatural transformations between them and 2-cells the modifications between those.
\end{Defi}
One must verify that this is actually a 2-category, but composition of 1-cells is inherited directly from the composition of both 1- and 2-cells in the target, hence is strictly associative and unital.

\begin{Defi}
The \emph{Gray tensor product}, $A \otimes B$, of a pair of 2-categories $A,B$ is a representing object for the functor $\twocat(A,Ps(B,-)): \twocat \to \Set$.
\end{Defi}

As a representing object, $A \otimes B$ comes with an isomorphism
\[
\twocat(A \otimes B,C) \cong \twocat(A,Ps(B,C))
\]
natural in $C$.  To verify that such a representation of $A \otimes B$ indeed exists one can describe a presentation of it in terms of generators and relations.  In the lax setting such a presentation was first described in Theorem I.4.9 of \cite{Gray1974Formal}.  A detailed presentation in the pseudo case of interest here is given in Section 3.1 \cite{Gurski2013Coherence}.  For completeness we give an example of an argument avoiding presentations.  We will use the following fact.

\begin{lem}\label{lem:left_adj_lfp} Let $U:A \to B$ a functor between locally presentable categories and let the set of objects $\{b_{i} \in B:i \in I\}$ form a strong generator in $B$.  Then the following are equivalent.
\begin{enumerate}
\item $U$ has a left adjoint.
\item Each $B(b_{i},U-)$ is representable.
\item Each $B(b_{i},U-)$ has a left adjoint.
\end{enumerate}
\end{lem}

For the uninitiated let us point out that locally presentable categories capture those categories of structures describable using limit-theories: so small categories or 2-categories, for instance, both of which can be defined using only pullbacks.  A standard reference is \cite{Adamek1994Locally}. A set of objects $\{b_{i}:i \in I\}$ is said to form a \emph{strong generator} if the representable functors $B(b_{i},-)$ are jointly conservative.  In $\twocat$ the three objects $\{S_{i}:i =0,1,2\}$  consisting of the free 0-cell, 1-cell and 2-cell form a strong generator.

\begin{proof}[Proof of \cref{lem:left_adj_lfp}]
We outline the straightforward proof.  Certainly $(1 \implies 2)$ by the definition of an adjunction whilst $(2 \iff 3)$ since $B$ is cocomplete.  For $(2 \implies 1)$ the main point is that if $F,G$ and $H$ are functors between locally presentable categories with $FG = H$, $H$ a right adjoint and $F$ a conservative right adjoint, then $G$ is a right adjoint.  This follows from the fact that right adjoints between locally presentable categories are the limit preserving functors that preserve $\lambda$-filtered colimits for some regular cardinal $\lambda$, together with the fact that conservative functors reflects any limits or colimits that they preserve.  Then take $F=B(b,-):B \to \Set^{I}$ and $G=U$.
\end{proof}

\begin{prop}\label{prop:otimes_exists}
The functor $Ps(A,-)$ has a left adjoint for all 2-categories $A$, and hence the Gray tensor product $A \otimes B$ exists.
\end{prop}
\begin{proof}
In order to show that $Ps(A,-)$ has a left adjoint it suffices, by \cref{lem:left_adj_lfp}, to show that each $\twocat(S_{i},Ps(A,-)):\twocat \to \Set$ is representable.  Now it is straightforward to show that there is a natural isomorphism 
\[
\twocat(B,Ps(A,-)) \cong \twocat(A,Ps(B,-))
 \]
 for all $B$.  Accordingly we must show that $\twocat(A,Ps(S_{i},-)):\twocat \to \Set$ is representable or, equivalently, has a left adjoint.  By the lemma this is true just when each $\twocat(S_{j},Ps(S_{i},-)):\twocat \to \Set$ is representable for $i,j \in \{0,1,2\}$.  Each of the nine cases is obvious.  For $i=0$ we have $\twocat(S_{j},Ps(S_{0},-)) \cong \twocat(S_{j},-)$.  For $i=1$ the representing objects are $S_{1}$, the pseudo-commutative square as depicted left below
$$\xy
(0,0)*+{a}="00";(15,0)*+{b}="10";(8,-7)*+{\cong^{\theta}};
 (0,-15)*+{c}="01";(15,-15)*+{d}="11";
{\ar^{r} "00"; "10"};
{\ar_{f} "00"; "01"};
{\ar_{s} "01"; "11"};
{\ar^{g} "10"; "11"};
\endxy
\hspace{2.5cm}
\xy
(0,0)*+{a}="00";(15,0)*+{b}="10";(12,-7)*+{\cong^{\theta^{\prime}}};
 (0,-15)*+{c}="01";(15,-15)*+{d}="11";
{\ar^{r} "00"; "10"};
{\ar@/_1pc/_{f} "00"; "01"};
{\ar@/^1pc/^{f^{\prime}} "00"; "01"};
{\ar_{s} "01"; "11"};
{\ar@/^1pc/^{g^{\prime}} "10"; "11"};
{\ar@{=>}^{\alpha}(-2,-7)*{};(2,-7)*{}};
\endxy
\hspace{0.1cm}
\xy
(0,-8)*+{=};
\endxy
\hspace{0.1cm}
\xy
(0,0)*+{a}="00";(15,0)*+{b}="10";(4,-7)*+{\cong^{\theta}};
 (0,-15)*+{c}="01";(15,-15)*+{d}="11";
{\ar^{r} "00"; "10"};
{\ar@/_1pc/_{f} "00"; "01"};
{\ar_{s} "01"; "11"};
{\ar@/_1pc/_{g} "10"; "11"};
{\ar@/^1pc/^{g^{\prime}} "10"; "11"};
{\ar@{=>}^{\beta}(13,-7)*{};(17,-7)*{}};
\endxy$$
and the 2-category depicted, in its entirety, above right.  The case $i=2$ is left to the interested reader.
\end{proof}

Each 2-natural transformation can be viewed as a pseudonatural transformation with identity 2-cell components, and each pseudonatural transformation can be viewed as a transformation by forgetting its 2-cell components.  Correspondingly, we have a commutative triangle of natural maps
\[
\xy
(0,0)*+{[B,C]}="00"; (30,0)*+{Ps(B,C)}="10";
(30,-15)*+{[B,C]_{f}}="11";
{\ar^{J_{1}} "00";"10"};
{\ar ^{J_{2}}"10";"11"};
{\ar _{J}"00";"11"};
\endxy
\]
in which the composite is postcomposition by the inclusion $J$ considered in \cref{eq:j}.  At a 2-category $C$ we thus obtain the top row of
\[
\cd{\twocat(A,[B,C]) \ar[d] \ar[r]^{(J_{1})_{*}} & \twocat(A,Ps(B,C)) \ar[d] \ar[r]^{(J_{2})_{*}} & \twocat(A,[B,C]_{f}) \ar[d]\\
\twocat(A \times B,C) \ar[r]_{Q^{*}} & \twocat(A \otimes B,C) \ar[r]_{P^{*}} & \twocat(A \star B,C)}
\]
in which the vertical arrows are isomorphisms natural in $C$.  On the bottom row are the unique functions making the diagram commute.  Since these are natural in $C$ they are uniquely induced by maps $P:A \star B \to A \otimes B$ and $Q:A \otimes B \to A \times B$.  Since the composite bottom row is given by precomposition with $K$ we obtain a factorisation
\begin{equation}\label{eq:PQ}
\xy
(0,0)*+{A \star B}="00"; (30,0)*+{A \otimes B}="10";
(30,-15)*+{A \times B}="11";
{\ar^{P} "00";"10"};
{\ar ^{Q}"10";"11"};
{\ar _{K}"00";"11"};
\endxy
\end{equation}
of $K:A \star B \to A \times B$.  It is this factorisation that we will show fits into the framework of \cref{thm:factormonoidal}, hence giving a new proof that the Gray tensor product is part of a closed symmetric monoidal structure.

\section{Factorisation systems on \twocat}
The aim of this section is to study the 2-functors $P:A \star B \to A \otimes B$ and $Q:A \otimes B \to A \times B$, and to show that they belong, respectively, to the left and right classes of a factorisation system on \twocat.  The factorisation system on \twocat that we will use is the $(boba/lff)$-factorisation system.  The left class consists of those 2-functors which induce isomorphisms on underlying categories -- are \textbf{b}ijective on \textbf{o}bjects and \textbf{b}ijective on \textbf{a}rrows -- whilst the right class consists of those $F:A \to B$ for which each functor $F_{x,y}:A(x,y) \to B(Fx,Fy)$ is full and faithful.  We leave it to the reader to verify that these classes of maps constitute an orthogonal factorisation system.

The two classes of a factorisation system $(\m E, \m M)$ on $C$ always satisfy certain stability properties.  One property is that whenever $gf \in \m E$ and $f \in \m E$ then $g \in \m E$.  Sometimes, as in the above example, one also has that $gf \in \m E$ and $g \in \m E$ imply $f \in \m E$, so that the left class satisfies \emph{2 from 3.}

\begin{Defi}
A factorisation system $(\m E, \m M)$ is a \emph{reflective factorisation system} \cite{Cassidy1985Reflective} if $ C$ has a terminal object $1$ and the maps in $\m E$ satisfy 2 from 3.
\end{Defi}

The important point for us is that in such a system, membership of $\m E$ can be tested more easily.
\begin{lem}\label{lem:reflective_fs}
Let $(\m E, \m M)$ be a reflective factorisation system on $ C$ and $f:a \to b$ be a morphism  in $ C$.  Then $f \in \m E$ so long as $ C(f,x):C(b,x) \to  C(a,x)$ is invertible for each $x$ such that $x \to 1 \in \m M$.
\end{lem}
\begin{proof}
In \cite{Cassidy1985Reflective} $\mathring{E}$ denotes the class of morphisms $f$ for which $ C(f,x):C(b,x) \to  C(a,x)$ is invertible whenever $x \to 1 \in \m M$.  Theorem 2.3 of \emph{ibid.} establishes that $\mathring{E} = E$ whenever $E$ satisfies 2 from 3.
\end{proof}

The $(boba/lff)$-factorisation system on \twocat is reflective: isomorphisms satisfy 2 from 3, so 2-functors which are isomorphisms on underlying categories do as well.

\begin{Defi}
We call a 2-category $C$ \emph{locally contractible} if $C \to 1$ is lff.
\end{Defi}

$C$ being locally contractible is equivalent to saying that, given parallel 1-cells $f,g:x \rightrightarrows y \in C$, there exists a unique 2-cell $f \Rightarrow g$.  Each such 2-cell is necessarily invertible.

\begin{lem}\label{lem:loc_con}
If $C$ is locally contractible then $J_{2}:Ps(B,C) \to [B,C]_{f}$ is invertible.
\end{lem}
\begin{proof}
Consider 2-functors $F$ and $G$ and a transformation $\eta:F \to G$ between them in $[B,C]_{f}$.  Then at $f:x \to y \in B$ the square
\[
\cd{Fx \ar[d]_{Ff} \ar[r]^{\eta_{x}} & Gx \ar[d]^{Gf} \\
Fy \ar[r]_{\eta_{y}} & Gy \\
}
\]
need not commute, but since $C$ is locally contractible there exists a unique isomorphism $\eta_{f}:Gf \circ \eta_{X} \cong \eta_{Y}\circ Ff$.  Now because all diagrams of 2-cells in $C$ commute it follows that the equations for a pseudonatural transformation cannot help but hold.
\end{proof}

\begin{prop}\label{prop:pboba}
The map $P:A \star B \to A \otimes B$ is boba.
\end{prop}
\begin{proof}
Let $C$ be locally contractible.  By \cref{lem:reflective_fs} to show that $P$ is boba it suffices to show that $P^{\star}:\twocat(A \otimes B,C) \to \twocat(A \star B,C)$ is invertible.  By adjointness this is equally to say that $(J_{2})_{\star}:\twocat(A,Ps(B,C)) \to \twocat(A,[B,C]_{f})$ is invertible, so it will suffice to show that $J_{2}:Ps(B,C) \to [B,C]_{f}$ is invertible, which follows by \cref{lem:loc_con}.
\end{proof}


Our aim now is to show that the 2-functor $Q:A \otimes B \to A \times B$ is locally full and faithful.  We will start by describing a section of it: the universal cubical functor $R:A \times B \rightsquigarrow A \otimes B$.  We then prove that $R$ is an equivalence inverse to $Q$ in the 2-category $\Icon$ and conclude that $Q$ is, in particular, locally full and faithful.

In the following we will consider pseudofunctors between 2-categories.  We will denote a pseudofunctor by a squiggly arrow such as $F:A \rightsquigarrow B$.  Recall that a pseudofunctor need not preserve composition or units on the nose but only up to invertible 2-cells $F(gf)\cong FgFf$ and $F1_{x} \cong 1_{Fx}$ satisfying coherence axioms which are essentially the same as the oplax monoidal functor axioms; these were first described in \cite{Benabou1967Introduction}.

\begin{Defi}\label{defn:cubical_fun}
A \emph{cubical functor} $F:A \times B \rightsquigarrow C$ is a normal (i.e., strictly identity-preserving) pseudofunctor such that for all pairs of composable arrows in $A \times B$
\[
\cd{(a_{1},a_{2}) \ar[r]^{(f_{1},f_{2})} & (b_{1},b_{2}) \ar[r]^{(g_{1},g_{2})} & (c_{1},c_{2})}
\]
with the property that if $f_{1}$ or $g_{2}$ is identity, the comparison 2-cell
\[F((g_{1},g_{2})\circ(f_{1},f_{2})) \cong F(f_{1},f_{2}) \circ F(g_{1},g_{2})\]
 is an identity.
 \end{Defi}

\begin{comment}
This is, as is probably clear, an obscure notion.  It is a translation of Gray's cubical functors \cite{Gray1974Formal} -- not defined as a special kind of pseudofunctor -- into the language of pseudofunctors.  It is justified by the following result (4.7 of \cite{Gordon1995Coherence}).
\end{comment}

\begin{nota}
For 2-categories $A, B, C$, let $Cub(A \times B, C)$ denote the set of cubical functors $A \times B \to C$.
\end{nota}

\begin{prop}
There is a bijection $Cub(A \times B,C) \cong \twocat(A,Ps(B,C))$, natural in $C$, which extends the bijection $\twocat(A \times B,C) \cong \twocat(A,[B,C])$
\end{prop}
\begin{proof}[Sketch]
The point is that evaluation gives a \emph{universal cubical functor}
\[
ev:Ps(B,C) \times B \rightsquigarrow C
\]
sending
\begin{itemize}
\item $(F,a)$ to $Fa$;
\item $(\eta,\alpha):(F,a) \to (G,b)$ to the composite $\eta_{b} \circ F\alpha:Fa \to Fb \to Gb$; and
\item at a composable pair
\[
\cd{(F,a) \ar[r]^{(\eta,\alpha)} & (G,b) \ar[r]^{(\mu,\beta)} & (H,c)}
\]
the comparison 2-cell is given by the isomorphism
\[
\xy
(-20,-00)*+{Fa}="-10";
(0,0)*+{Fb}="00"; (20,10)*+{Fc}="21";  (20,-10)*+{Gb}="2-1";
(40,0)*+{Gc}="30";(60,0)*+{Hc}="40";
{\ar^{F\alpha} "-10";"00"};
{\ar^{F\beta}"00";"21"};
{\ar^{\eta_{c}}"21";"30"};
{\ar@{=>}^{\eta_{\beta}}(20,4)*+{};(20,-4)*+{}};
{\ar^{\mu_{c}}"30";"40"};
{\ar_{\eta_{b}}"00";"2-1"};
{\ar_{G\beta}"2-1";"30"};
\endxy
\]
which is an identity whenever $\eta=1$ or $\beta=1$.
\end{itemize}  The remainder is a routine calculation.
\end{proof}

In particular we obtain a composite isomorphism
\[
\twocat(A \otimes B,C) \cong \twocat(A,Ps(B,C)) \cong Cub(A \times B,C)
\]
natural in $C$.  The Yoneda lemma therefore yields a \emph{universal cubical functor}
\[
R:A \times B \rightsquigarrow A \otimes B,
\]
precomposition with which induces a bijection (natural in $C$) between 2-functors $A \otimes B \to C$ and cubical functors $A \times B \rightsquigarrow C$.  In particular the identity 2-functor $A \times B \to A \times B$ factors uniquely through $R$ as a 2-functor $A \otimes B \to A \times B$: our $Q$ from before, as we now show.

\begin{prop}
The 2-functor $Q:A \otimes B \to A \times B$ of \cref{eq:PQ} satisfies $QR=1$.
\end{prop}
\begin{proof}
By definition $R^{*}$ fits into the commutative diagram
\[
\cd{\twocat(A \times B,C) \, \ar[d]_{\cong} \ar[r]^{Q^{*}} & \twocat(A \otimes B,C) \ar[d]^{\cong} \\
\twocat(A,[B,C]) \, \ar[d]_{\cong} \ar@{->}[r]^{(J_{1})_{*}} & \twocat(A,Ps(B,C)) \ar[d]^{\cong}  \\
\twocat(A \times B,C) \, \ar@{->}[r] & Cub(A \times B,C)}
\]
as the composite of the right vertical leg; here the vertical arrows are the natural isomorphisms from various adjunctions and the bottom horizontal arrow is the inclusion.
Setting $C=A \times B$ and chasing the identity $1:A \times B \to A \times B$ from top left to bottom right gives the result.
\end{proof}

\begin{Defi}
The collection of 2-categories, pseudofunctors and icons (\cref{Defi:icon}) forms a strict 2-category which we will call \Icon, and we write $\Icon_{s}$ for the sub 2-category containing the 2-functors.
\end{Defi}

Since \Icon is a 2-category one may speak of an equivalence therein: an arrow $F:A \rightsquigarrow B \in \Icon$ such that there exists $G:B \rightsquigarrow A$ and 2-cells $FG \cong 1$ and $GF \cong 1$.  Let us record the following straightforward characterisation of the equivalences in $\Icon$.  We note that a similar though more involved result for lax functors is established in Proposition 4.3 of \cite{Lack2010Icons}.

\begin{prop}\label{prop:equiv_ff}
A 2-functor $F:A \to B$ is an equivalence in $Icon$ just when it is bijective on objects (bo) and each $A(x,y) \to B(Fx,Fy)$ is an equivalence of categories.  In particular $F$ is then locally full and faithful.
\end{prop}

Because $\Icon$ is a 2-category we can consider 2-categorical limits in it.  The crucial limit is the pseudolimit of an arrow.  We will use these to establish that the universal cubical functor $Q:A \times B \rightsquigarrow A \otimes B$ is equivalence inverse to $P:A \otimes B \to A \times B$, employing a similar trick in \cref{prop:qlff} to that in Theorem 4.2 of \cite{Blackwell1989Two-dimensional}.

\begin{Defi}\label{Defi:pslim}
Let $f: a \to b$ be a 1-cell in a 2-category $K$.  The \emph{pseudolimit} of $f$ is an object $l$ together with a pair of 1-cells $u,v$ and a 2-cell $\lambda$
\[
\xy
(0,0)*+{l}="1";
(-12,-12)*+{a}="2";
(12,-12)*+{b}="3";
(0,-8)*{\cong \lambda};
{\ar_{u} "1"; "2"};
{\ar^{v} "1"; "3"};
{\ar_{f} "2"; "3"};
\endxy
\]
satisfying the following universal property.
\begin{enumerate}
\item Given any other object $t$ together with 1-cells $p,q$ and 2-cell $\tau$ as below on the left,
\[
\xy
(0,0)*+{
\xy
(0,0)*+{t}="1";
(-12,-20)*+{a}="2";
(12,-20)*+{b}="3";
(0,-13.2)*{\cong \tau};
{\ar_{p} "1"; "2"};
{\ar^{q} "1"; "3"};
{\ar_{f} "2"; "3"};
\endxy};
(20,0)*{=};
(40,0)*+{
\xy
(0,0)*+{t}="0";
(0,-10)*+{l}="1";
(-12,-20)*+{a}="2";
(12,-20)*+{b}="3";
(0,-16.6)*{\cong \lambda};
{\ar_{u} "1"; "2"};
{\ar^{v} "1"; "3"};
{\ar_{f} "2"; "3"};
{\ar^{k} "0"; "1"};
\endxy};
\endxy
\]
there exists a unique 1-cell $k:t \to l$ such that $p = uk, q = vk, \tau  = \lambda * k$ as above on the right.
\item Given two such $(t,p,q,\tau), \, (t,p',q',\tau')$ sharing the same source object $t$ and with corresponding 1-cells $k, k'$, and 2-cells $\rho: p \Rightarrow p', \sigma: q \Rightarrow q'$ such that
\[
(f * \rho) \circ \tau = \tau' \circ \sigma,
\]
there exists a unique 2-cell $\theta: k \Rightarrow k'$ such that $\rho = u * \theta, \sigma = v * \theta$.
\end{enumerate}
\end{Defi}
Given a pseudofunctor $F:A \rightsquigarrow B$ its pseudolimit in \Icon
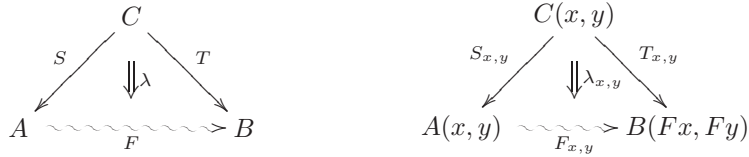
\[
\xy
(0,0)*+{A}="00"; (15,15)*+{C}="11"; (30,0)*+{B}="20";
{\ar_{S} "11";"00"};{\ar^{T} "11";"20"};{\ar@{~>}_{F}"00";"20"};
{\ar@{=>}^{\lambda}(15,10)*+{};(15,3)*+{}};
\endxy
\hspace{2cm}
\xy
(0,0)*+{A(x,y)}="00"; (15,15)*+{C(x,y)}="11"; (30,0)*+{B(Fx,Fy)}="20";
{\ar_{S_{x,y}} "11";"00"};{\ar^{T_{x,y}} "11";"20"};{\ar@{~>}_{F_{x,y}}"00";"20"};
{\ar@{=>}^{\lambda_{x,y}}(15,10)*+{};(15,3)*+{}};
\endxy
\]
is described as follows\begin{footnote}{We note that the existence of such pseudolimits in \Icon can alternatively be established using results of \cite{Blackwell1989Two-dimensional} on 2-monads, upon observing that \Icon is the 2-category of algebras and pseudomorphisms for the free 2-category monad on the 2-category of \Cat-enriched graphs.}\end{footnote}.  $C$ has the same objects as $A$; each hom-category $C(x,y)$ is constructed as the pseudolimit of the functor $F_{x,y}:A(x,y) \to B(Fx,Fy)$ in \Cat  as depicted in the diagram above right.  So $S$ is the identity on objects and $T$ and $F$ agree on objects.

More concretely, an arrow $x \to y$ of $C$ is given by a triple $(f,\theta,g)$ with
\begin{itemize}
\item $f:x \to y \in A$ and
\item an invertible 2-cell $\theta:g \cong Ff$.
\end{itemize}
  A 2-cell $(f_{1},g_{1},\theta_{1}) \Rightarrow (f_{2},g_{2},\theta_{2})$ consists of a pair of 2-cells $\beta:f_{1} \Rightarrow f_{2}$ and $\alpha:g_{1} \Rightarrow g_{2}$ making the square
\[
\cd{g_{1} \ar@{=>}[d]_{\alpha} \ar@{=>}[r]^{\theta_{1}} &  Ff_{1} \ar@{=>}[d]^{F\beta} \\
 g_{2} \ar@{=>}[r]_{\theta_{2}} &  Ff_{2}\\
 }
\]
commute.  Given $(f_{1},\theta_{1},g_{1}):x \to y$ and $(f_{2},\theta_{2},g_{2}):y \to z$ the composite is given by $(f_{2}f_{1},\phi,g_{2}g_{1})$ where $\phi$ is the composite
\[
g_{2}g_{1} \cong Ff_{2}Ff_{1} \cong F(f_{2}f_{1}).
\]
  Composition of 2-cells is clear.  Then the 2-functors $S$ and $T$ are the evident projections to $A$ and $B$ whilst the icon $\lambda$ has component at $(f,\theta,g):x \to y$ given by $\theta:g \cong Ff$ itself.  That $C$ has the claimed universal property is easily verified.  An obvious yet crucial fact is that a 2-cell of $C$ is an identity just when its images under $S$ and $T$ are identities.  We summarise below.

\begin{prop}
The pseudolimit $(C,S,\lambda,T)$ of an arrow in $\Icon$ exists and has projections $S$ and $T$ given by 2-functors.  The projections $S$ and $T$ jointly reflect identity 2-cells.
\end{prop}

\begin{prop}\label{prop:qlff}
The comparison $Q:A \otimes B \to A \times B$ is an equivalence in $\Icon$.  In particular it is locally full and faithful.
\end{prop}
\begin{proof}
Consider the universal cubical functor $R:A \times B \rightsquigarrow A \otimes B$.  We form its pseudolimit $(C, \lambda)$ in $\Icon$ as in the interior triangles below.
\[
\xy
(-20,-00)*+{A \times B}="-10";
(0,0)*+{C}="00"; (20,10)*+{A \times B}="10";
(20,-10)*+{A \otimes B}="11";
{\ar^{S} "00";"10"};
{\ar@{~>}^{R}"10";"11"};
{\ar_{T}"00";"11"};
{\ar@{=>}^{\lambda}(17,4)*+{};(12,-5)*+{}};
{\ar@{~>}^{U}"-10";"00"};
{\ar@/^1pc/^{1}"-10";"10"};
{\ar@/_1pc/@{~>}_{R}"-10";"11"};
\endxy
\hspace{0.2cm}
\xy
(-38,0)*+{A \times B}="-20";
(-20,-00)*+{A \otimes B}="-10";
(0,0)*+{C}="00"; (20,10)*+{A \times B}="10";
(20,-10)*+{A \otimes B}="11";
{\ar^{S} "00";"10"};
{\ar@{~>}^{R}"10";"11"};
{\ar_{T}"00";"11"};
{\ar@{=>}^{\lambda}(17,4)*+{};(12,-5)*+{}};
{\ar^{L}"-10";"00"};
{\ar@/^1pc/^{Q}"-10";"10"};
{\ar@/_1pc/_{1}"-10";"11"};
{\ar@{~>}^{R}"-20";"-10"};
\endxy
\]

The universal property of $C$ induces a unique pseudofunctor $U:A \times B \rightsquigarrow C$ such that  $SU=1$, $TU=R$ and $\lambda U =1$.  Because both $SU=1$ and $TU=R$ are cubical functors, and because $S$ and $T$ jointly reflect the property of a 2-cell being an identity, it follows that $U$ is also cubical.  As such we obtain a unique 2-functor $L:A\otimes B \to C$ such that $LR=U$.  The universal property of $R:A \times B \rightsquigarrow A \otimes B$ ensures that $SL=Q$ and $TL=1$ and we obtain the desired invertible icon as the horizontal composite $\lambda L:RQ \cong 1$.  The final claim then follows from \cref{prop:equiv_ff}.
\end{proof}

\begin{thm}
For 2-categories $A, B$, the factorisation
\[
A \star B \stackrel{P}{\longrightarrow} A \otimes B \stackrel{Q}{\longrightarrow} A \times B
\]
of the canonical comparison 2-functor $K:A \star B \to A \times B$ is the factorisation of $K$ into a boba 2-functor followed by an lff 2-fuctor.  Furthermore, $\otimes$ can be extended to a unique closed symmetric monoidal structure on \twocat such that both $P$ and $Q$ are oplax monoidal structures on the identity functor.
\end{thm}
\begin{proof}
The first claim is \cref{prop:pboba} and \cref{prop:qlff}.  The final claim then follows from \cref{thm:factormonoidal}, with the only thing to check is that $\star$ preserves boba 2-functors and $\times$ preserves lff 2-functors, both of which are straightfoward to verify.
\end{proof}

\bibliographystyle{acm}

\end{document}